\documentclass[a4paper,11pt]{amsart}
\usepackage{hyperref,fancyhdr,mathrsfs,amsmath,amscd,amsthm,amsfonts,latexsym,amssymb,stmaryrd}
\usepackage[all]{xy}

\newcommand{\ol}{\mathcal{O}}

\def \phi{\varphi}
\def \Phi{\varPhi}
\def \p{\pi}

\def \t{\tau}

\def \Hq{\mathbb{H}\,}
\def \C{\mathbb{C}\,}

\def\widecheckg{g^{\hspace*{-2.5pt}\vbox to 5pt{\hbox to
0pt{\LARGE$\check{}$}}}\hspace*{2pt}}

\def\widecheckl{\lambda^{\hspace*{-3.5pt}\vbox to 8pt{\hbox to
0pt{\LARGE$\check{}$}}}\hspace*{2pt}}

\begin{document}

\title{On the quaternionic manifolds whose\\
twistor spaces are Fano manifolds}
\author{Radu Pantilie} 
\thanks{The author acknowledges partial financial support from the Romanian National Authority for Scientific Research, CNCS-UEFISCDI, project no.\ 
PN-II-ID-PCE-2011-3-0362.} 
\email{\href{mailto:radu.pantilie@imar.ro}{radu.pantilie@imar.ro}}
\address{R.~Pantilie, Institutul de Matematic\u a ``Simion~Stoilow'' al Academiei Rom\^ane,
C.P. 1-764, 014700, Bucure\c sti, Rom\^ania}

\newtheorem{thm}{Theorem}[section]
\newtheorem{lem}[thm]{Lemma}
\newtheorem{cor}[thm]{Corollary}
\newtheorem{prop}[thm]{Proposition}

\theoremstyle{definition}

\newtheorem{defn}[thm]{Definition}
\newtheorem{rem}[thm]{Remark}
\newtheorem{exm}[thm]{Example}

\numberwithin{equation}{section}

\begin{abstract}
Let $M$ be a quaternionic manifold, $\dim M=4k$\,, whose twistor space is a Fano manifold. We prove the following:\\
\indent
(a) $M$ admits a reduction to ${\rm Sp}(1)\times{\rm GL}(k,\Hq)$ if and only if $M=\Hq\!P^k$,\\
\indent
(b) either $b_2(M)=0$ or $M={\rm Gr}_2(k+2,\C)$\,.\\
\noindent
This generalizes results of S.~Salamon and C.~R.~LeBrun, respectively, who obtained the same conclusions under the assumption
that $M$ is a complete quaternionic-K\"ahler manifold with positive scalar curvature.
\end{abstract}

\maketitle
\thispagestyle{empty}
\vspace{-8mm}

\section{Introduction}

An \emph{almost quaternionic structure} on a manifold $M$ is a reduction of its frame bundle to ${\rm Sp}(1)\cdot{\rm GL}(k,\Hq)$\,.
Then the obstruction for $M$ to admit a `reduction' to ${\rm Sp}(1)\times{\rm GL}(k,\Hq)$ is an element of $H^2(M,\mathbb{Z}_2)$ \cite{MarRom-r}\,.
Equivalently, this is the second Stiefel-Whitney class of the oriented Riemannian vector bundle $Q$ induced by the Lie groups morphism
${\rm Sp}(1)\cdot{\rm GL}(k,\Hq)\to{\rm SO}(3)$\,, $\pm(a,A)\mapsto\pm a$\,.\\
\indent
If $\dim M\geq8$ then the almost quaternionic structure is \emph{integrable} if there exists a torsion free connection on $M$ which is
compatible (with the structural group) \cite{Sal-dg_qm}\,. Equivalently (see \cite{IMOP}\,), there exists a compatible connection $\nabla$ on $M$
such that the almost complex structure induced by $\nabla$ on the sphere bundle $Z$ of $Q$ is integrable. Then the complex manifold $Z$
is the \emph{twistor space} of $M$ and the fibres of $\p:Z\to M$ are the `real' \emph{twistor lines}; furthermore, $Z$ is endowed with a conjugation
(given by the antipodal map on the fibres of $\p$). Conversely, $Z$ together with its conjugation and a real twistor line determines $M$
(see \cite{Pan-twistor_(co-)cr_q}\,). Furthermore, by \cite{Sal-dg_qm} and \cite{PePoSw-98}\,, there exists a holomorphic line bundle
$\mathcal{L}$ over $Z$ whose restriction to any twistor line has Chern number $2$\,. It follows quickly that $M$ admits a reduction to
${\rm Sp}(1)\times{\rm GL}(k,\Hq)$ if and only if $\mathcal{L}$ admits a square root.\\
\indent
Further natural restrictions can be obtained by assuming that there exists a Riemannian metric on $M$ for which the holonomy group of its
Levi-Civita connection is contained by ${\rm Sp}(1)\cdot{\rm Sp}(k)$\,; then $M$ is called \emph{quaternionic-K\"ahler}.
It follows \cite{Sal-Inventiones1982} that any quaternionic-K\"ahler manifold is an Einstein manifold, and, assuming, further, completeness
and the scalar curvature positive, the corresponding twistor space is a Fano manifold.
Also, by \cite[Theorem 6.3]{Sal-Inventiones1982}\,, $\Hq\!P^k$ is the only such quaternionic-K\"ahler manifold
which admits a reduction to ${\rm Sp}(1)\times{\rm GL}(k,\Hq)$\,.\\
\indent
Another result, in the same vein, is \cite{LeB-top_qK} that for any complete quaternionic-K\"ahler manifold $M$ with positive scalar curvature
we have that either its second Betti number $b_2(M)$ is zero, or $M$ is the Grassmannian ${\rm Gr}_2(k+2,\C)$\,, where $\dim M=4k$\,.\\
\indent
In this paper, we generalize these two results of \cite{Sal-Inventiones1982} and \cite{LeB-top_qK}\,, respectively, to the class of quaternionic manifolds
whose twistor spaces are Fano manifolds.\\
\indent
I am very grateful to Paltin Ionescu for an illuminating
exchange of e-mail messages, during which he generously provided, for example, the proof of Proposition \ref{prop:with_cc}\,. I am, also,
grateful to Liviu Ornea and to Victor Vuletescu for useful comments on a preliminary version of the paper.

\section{The results}

\indent 
As the four-dimensional case was elucidated in \cite{Hit-Kaehler_twist}\,, we consider only quaternionic manifolds of dimension 
at least $8$\,.\\  
\indent
The following result generalizes \cite[Theorem 6.3]{Sal-Inventiones1982}\,.

\begin{thm} \label{thm:qrFt}
Let $M$ be a quaternionic manifold, $\dim M=4k\geq8$\,, which admits a reduction to ${\rm Sp}(1)\times{\rm GL}(k,\Hq)$\,;
denote by $Z$ the twistor space of $M$.\\
\indent
Then the following assertions are equivalent:\\
\indent
\quad{\rm (i)} $M=\Hq\!P^k$;\\
\indent
\quad{\rm (ii)} $M$ is simply-connected, $b_2(M)=0$ and $Z$ is projective (that is, $Z$ can be embedded as a compact complex submanifold of a complex projective space);\\
\indent
\quad{\rm (iii)} $Z$ is a Fano manifold (that is, $Z$ is compact and its anticanonical line bundle is ample).
\end{thm}
\begin{proof}
It is obvious that if (i) holds then both (ii) and (iii) are statisfied, as $Z=\C\!P^{2k+1}$ and $M=\Hq\!P^k$.\\ 
\indent 
Further, as the restriction of the holomorphic cotangent bundle to each
twistor line is $\ol(-2)\oplus2k\ol(-1)$\,, where $\ol(-1)$ is the tautological line bundle, essentially the same proof as for  
\cite[Proposition 2.2(ii)]{Hit-Kaehler_twist} implies that any holomorphic form of positive degree on $Z$ is zero.
Consequently, if $Z$ is projective, from the exact sequence of cohomology groups associated to the exact sequence of complex Lie groups 
$0\to\mathbb{Z}\to\C\!\to\C\setminus\{0\}\to0$ (determined by the exponential) we deduce that the Picard group ${\rm Pic}(Z)$ is isomorphic 
to $H^2(Z,\mathbb{Z})$\,. Furthermore, if (ii) holds then, also, $Z$ is simply-connected (by the homotopy exact sequence determined 
by the smooth bundle $Z\to M$), and, hence, ${\rm Pic}(Z)$ has no torsion. Also, as $b_2(Z)=b_2(M)+1$ (see \cite{LeB-top_qK}\,), 
${\rm Pic}(Z)$ has rank $1$. We have, thus, proved that ${\rm Pic}(Z)$ is isomorphic to $\mathbb{Z}$\,.\\ 
\indent 
Let $\mathcal{L}$ be the restriction to $Z$ of the dual of the tautological line bundle over the complex projective space in which $Z$ is embedded.
As both the restriction of $\mathcal{L}$ and of the anticanonical line bundle $K_Z^*$ of $Z$, to a twistor line, are positive we deduce that
$\bigl(K_Z^*\bigr)^p=\mathcal{L}^q$, for some positive integers $p$ and $q$\,. Thus, also $\bigl(K_Z^*\bigr)^p$ is very ample, and (ii)$\Longrightarrow$(iii) is proved.\\
\indent
To complete the proof it is sufficient to show that (iii)$\Longrightarrow$(i)\,.
We claim that, if (iii) holds, there exists a holomorphic line bundle $\mathcal{L}$ over $Z$ such that:\\
\indent
\quad(a) $\mathcal{L}$ is ample;\\
\indent
\quad(b) $\mathcal{L}$ restricted to each twistor line is (isomorphic to) $\ol(1)$\,.\\
\indent
Indeed, from the assumption that $M$ admits a reduction to ${\rm Sp}(1)\times{\rm GL}(k,\Hq)$\,, 
by \cite{Sal-dg_qm} and \cite{PePoSw-98} there exists a holomorphic line bundle $\mathcal{L}_1$ over $Z$
which satisfies condition (b)\,, above; moreover, $\mathcal{L}_1$ is endowed with a morphism of (real) vector bundles whose square is $-1$
and which is an anti-holomorphic diffeomorphism covering the conjugation of $Z$ (given, on each fibre of $Z\to M$, by the antipodal map).
We shall show that after tensorising, if necessary, $\mathcal{L}_1$ with a
holomorphic line bundle, whose restriction to each twitor line is trivial, we obtain a line bundle satisfying (a)\,.\\
\indent
For this, firstly, note that $K_Z^*\,(=\Lambda_{\C\!}^{2k+1}TZ)$ restricted to each twistor line is $\ol(2k+2)$\,.
Hence, $K_Z\otimes\mathcal{L}_1^{2k+2}$ restricted to each twistor line is trivial; moreover, this holomorphic line bundle is endowed with a conjugation
(that is, an involutive morphism of vector bundles which is an anti-holomorphic diffeomorphism) covering the conjugation of $Z$.
Therefore $K_Z\otimes\mathcal{L}_1^{2k+2}$ corresponds, through the Ward transform, to a (real) line bundle $L$ over $M$
endowed with an anti-self-dual connection (that is, a connection whose curvature form is such that its $(0,2)$-part,
with respect to any admissible linear complex structure on $M$, is zero).\\
\indent
As $M$ is simply-connected (because $Z$ is Fano and therefore simply-connected, and the fibres of the projection $Z\to M$ are connected),
$L$ is orientable and, hence, there exists a line bundle $L_1$
such that $L=L_1^{2k+2}$; furthermore, this isomorphism is connection preserving with respect to a unique anti-self-dual connection on $L_1$\,.
Hence, $L_1$ corresponds to a holomorphic line bundle $\mathcal{L}_2$ over $Z$ whose restriction to each twistor line is trivial, and such that
$K_Z\otimes\mathcal{L}_1^{2k+2}=\mathcal{L}_2^{2k+2}$.\\
\indent
Thus, since $K_Z^*$ is ample, $\mathcal{L}=\mathcal{L}_1\otimes\mathcal{L}_2^*$ satisfies (a) and (b)\,, above. Moreover, $\mathcal{L}$
is endowed with a morphism of vector bundles $\t$ whose square is $-1$ and which is an anti-holomorphic diffeomorphism covering the conjugation of $Z$.
Hence, $\t$ induces a linear complex structure $J$ on $H^0(Z,\mathcal{L})$ which anti-commutes with its canonical complex structure.\\
\indent
By \cite[Corollary 2.4]{Paltin-2005}\,, $Z$ is a complex projective space and the twistor lines are just the complex projective lines;
moreover, $Z$ is the projectivisation of the dual of $H^0(Z,\mathcal{L})$\,. Furthermore, $J$ induces on the dual $E$ of $H^0(Z,\mathcal{L})$ 
a linear quaternionic structure with respect to which the fibres of $Z\to M$ are those complex projective lines obtained through the complex projectivisation 
of the quaternionic vector subspaces of $E$ of real dimension $4$\,. Thus, $Z=PE$, $M$ is the quaternionic projective space $P_{\Hq\!}E$,
and $Z\to M$ is the canonical projection $PE\to P_{\Hq\!}E$. The proof is complete.
\end{proof}

\indent
The following result generalizes \cite[Theorem 1]{LeB-top_qK}\,.

\begin{thm} \label{thm:qFt}
Let $M$ be a quaternionic manifold, $\dim M=4k\geq8$\,, whose twistor space is a Fano manifold.\\
\indent
Then either $b_2(M)=0$ or $M={\rm Gr}_2(k+2,\C)$\,.
\end{thm}
\begin{proof}
Let $Z$ be the twistor space of $M$. Similarly to the proof of Theorem \ref{thm:qrFt}\,, we obtain a holomorphic line bundle $\mathcal{L}$ over $Z$
such that $\mathcal{L}^{k+1}=K_Z^*$. Furthermore, $\mathcal{L}$ admits a square root if and only if $M$ admits a reduction to
${\rm Sp}(1)\times{\rm GL}(k,\Hq)$\,. Therefore, by Theorem \ref{thm:qrFt}\,, either $M=\Hq\!P^k$ or $k+1$ is the greatest natural number $n$
for which $K_Z^*$ admits a $n$-th root. From now on, in this proof, we shall assume that the latter holds.\\
\indent
Now, just like in the proof of \cite[Theorem 1]{LeB-top_qK}\,, by using \cite{Wies-2}\,, we obtain that if $b_2(M)\neq0$ then (at least)
one of the following three statements holds:\\
\indent
\quad(i) $Z=\C\!P^k\times Q_{k+1}$\,, where $Q_{k+1}$ is the nondegenerate hyperquadric in $\C\!P^{k+2}$,\\
\indent
\quad(ii) $Z$ is the projectivisation of the holomorphic cotangent bundle of $\C\!P^{k+1}$,\\
\indent
\quad(iii) $Z$ is $\C\!P^{2k+1}$ blown up along $\C\!P^{k-1}$.\\
\indent
The fact that (i) cannot occur is a consequence of Proposition \ref{prop:with_cc}\,, below.\\
\indent
In the remaining two cases, it follows that $M$ can be locally identified (through quaternionic diffeomorphisms) with
${\rm Gr}_2(k+2,\C)$ or with $\Hq\!P^k$, respectively. By using that $M$ is compact and simply-connected, a standard argument shows
that either $M={\rm Gr}_2(k+2,\C)$ or $M=\Hq\!P^k$. As the latter leads to a contradiction, the proof is complete.
\end{proof}

\indent
The following result, also interesting in itself, was used in the proof of Theorem \ref{thm:qFt}\,.

\begin{prop}[\cite{Paltin-emails}] \label{prop:with_cc}
Let $Q_{k+1}$ be the nondegenerate hyperquadric in $\C\!P^{k+2}$. Then no open subset of $\C\!P^k\times Q_{k+1}$ can be the twistor
space of a quaternionic manifold.
\end{prop}
\begin{proof}
We shall prove that $Y=\C\!P^k\times Q_{k+1}$ does not admit an embedded Riemann sphere whose normal bundle is $2k\ol(1)$\,.
Indeed, let $L_1$ and $L_2$ be the restrictions to $\C\!P^k$ and $Q_{k+1}$ of the duals of the tautological line bundles on $\C\!P^k$ and $\C\!P^{k+2}$,
respectively. We have that both $L_1$ and $L_2$ are very ample and, also, $K_{\C\!P^k}^*=(L_1)^{k+1}$\,, $K_{Q_{k+1}}^*=(L_2)^{k+1}$
(for the latter, use the adjunction formula mentioned in \cite[p.\ 147]{GriHar}\,).
Thus, on denoting by $\p_1$ and $\p_2$ the projections from $Y$ onto its factors, respectively, we obtain that, also,
$L=\p_1^*L_1\otimes\p_2^*L_2$ is very ample, and $K_Y^*=L^{k+1}$. Therefore if $Y$ would admit
an embedded Riemann sphere $t$ whose normal bundle is $2k\ol(1)$ then $L|_t=\ol(2)$\,.
On embedding $Y$ into the projectivisation of the dual of $H^0(Y,L)$\,, we obtain that $t$ has degree two and therefore it is a conic.
It follows that any two points of $Y$ are joined by a conic.
But, according to \cite{PIonRus-cc}\,, $Y$ cannot have this property, thus completing the proof.
\end{proof}

\end{document}